\documentclass[12pt,letterpaper]{amsart}
\usepackage{amsfonts, amsmath, amssymb, amscd, amsthm, graphicx}

\makeatletter
\def\blfootnote{\xdef\@thefnmark{}\@footnotetext}
\makeatother

\newtheorem{theorem}{Theorem}
\newtheorem{corollary}[theorem]{Corollary}
\newtheorem{lemma}[theorem]{Lemma}

\newtheorem{ques}[theorem]{Problem}

\theoremstyle{definition}
\newtheorem{definition}[theorem]{Definition}
\theoremstyle{remark}
\newtheorem{remark}[theorem]{Remark}
\newtheorem{example}[theorem]{Example}

\begin{document}

\title{Weakly amenable groups}

\author{Denis V. Osin}

\email{denis.osin@gmail.com}

\thanks{The work has been supported by the RFFR grant 99-01-00894
and by the Swiss National Science Foundation.}


\date{}

\begin{abstract}
We construct the first examples of finitely
generated non--amenable groups whose left regular representations
are not uniformly isolated from the trivial representation.
\end{abstract}

\maketitle


\section{Introduction.}


Recall that a locally compact group $G$ is called amenable if
there exists a finitely additive measure $\mu $ on the set of all
Borel subsets of $G$ which is invariant under the left action of
the group $G$ on itself and satisfies $\mu (G)=1$. The class of
amenable groups, $AG$, has been introduced by von Neumann
\cite{Neu} in order to explain the Hausdorff--Banach--Tarski
paradox and was investigated by a number of authors.

One of the most interesting characterizations of amenable groups
was obtained by Hulaniski \cite{H} in terms of
$L^2$--representations.

\begin{definition} One says that the left regular representation $L_G$
of a
locally compact group $G$ on the Hilbert space $L^2(G)$ {\it
weakly contains the trivial representation}, if for any
$\varepsilon >0$ and any compact subset $S\subseteq G$, there
exists $v\in L^2(G)$ such that $\|v\|=1$ and
\begin{equation}
\left| \langle v, sv\rangle -1\right| < \varepsilon \label{wc}
\end{equation}
for any $s\in S$.
\end{definition}

\begin{theorem}[Hulaniski]  A locally compact group $G$ is amenable
if and only if the left regular representation of $G$ weakly
contains the trivial representation.
\end{theorem}

Given a locally compact group $G$ and a compact subset $S\subseteq
G$, we define $\alpha (G, S)$ as the supremum of all $\varepsilon
\ge 0$ such that for any vector $v\in L^2(G)$ of norm $||v||=1$,
there exists an element $s\in S$ satisfying the inequality $$ ||
sv-v||\ge \varepsilon. $$ In case the group $G$ is discrete and
finitely generated, the existence of a finite generating set $S$
such that $\alpha (G, S)>0$, implies the inequality $\alpha (G,
S^\prime )>0$ for any other generating set $S^\prime $ of $G$.
Thus it is natural to consider the quantity $$\alpha (G)
=\inf\limits_{S} \alpha (G, S), $$ where $S$ ranges over all
finite generating sets of $G$. The following definition can be
found in \cite{Sh}

\begin{definition} The left regular representation of a finitely
generated
group $G$ is said to be {\it uniformly isolated from the trivial
representation} if $\alpha (G)>0$.
\end{definition}

Obviously one has
\begin{equation}
\alpha (G)=0 \label{1}
\end{equation}
for any finitely generated amenable group. Indeed, it is easy to
check that (\ref{wc}) implies $\| sv-v\| < \sqrt{2\varepsilon }$.
Thus (\ref{1}) follows from Theorem 1.2. On the other hand, it is
not clear whether the equality (\ref{1}) is equivalent to the
amenability of the group $G$. The following problem was suggested
by Shalom in \cite{Sh}.

\begin{ques}
Is the left regular representation of any non--amenable finitely
generated group uniformly isolated from the trivial
representation?
\end{ques}

In \cite{Sh}, the positive answer was obtained in the particular
case of residually finite hyperbolic groups. However, the question
remained open in general. The main purpose of the present note is
to show that the answer is negative and can be obtained by using
the methods developed in \cite{Osin}

The main part of this paper was written during the author's visit to
University of Geneva. I am grateful to Pierre de la Harpe for
invitation and constant attention to this work. Also I would like
to express my gratitude to Rostislav I. Grigorchuk, Anna G.
Erschler, Alexander Yu. Ol'shanskii, and the referee for useful
comments and remarks.

\section{Main results}

The main results of the paper are gathered in this section. We
call a finitely generated group $G$ {\it weakly amenable} if it
satisfies (\ref{1}) and denote by $WA$ the class of all weakly
amenable groups.

Two families of non--amenable weakly amenable groups are
constructed in the present paper. The idea of the first
construction is similar to one from  \cite{Gri-96}.

\begin{theorem} Let $A$ be a finitely generated abelian group. Suppose
that
there exist two monomorphisms $\lambda, \mu :A\to A$ with the
following properties.

1) $\lambda \circ \mu \equiv \mu \circ \lambda.$

2) The subgroup generated by $\lambda (A)\cup \mu (A)$ coincides
with $A$.

3) $\lambda (A)\cup \mu (A)\ne A$.

\noindent Then the HNN--extension
\begin{equation}
G=\langle A, t \; : \; t^{-1}\lambda (a)t=\mu (a), \; a\in
A\rangle \label{GnH}
\end{equation}
is a finitely generated weakly amenable non--amenable group.
\end{theorem}

\begin{example} Suppose that $A=\mathbb Z$ and $\lambda , \mu $ are
defined
by $\lambda (1)=m$, $\mu (1)=n$. If $m,n$ are relatively prime,
and $|m|\ne 1, |n|\ne 1$, one can easily verify the conditions of
Theorem 2.1. Taking the HNN--extension, we obtain the
Baumslag--Solitar group $$BS(m,n)=\langle a,t\; :\;
t^{-1}a^mt=a^n\rangle .$$ Using the Britton lemma on
HNN--extensions \cite[Ch. IV, Sec. 2]{LS}, one can prove that the
elements $t$ and $a^{-1}ta$ generates a free subgroup of rank $2$.
This shows that the class $WA$ is not closed under the taking of
subgroups.
\end{example}

In the last section of the present paper we give another way to
construct a weakly amenable non--amenable group using limits of
hyperbolic groups. The proof involves the tools of hyperbolic
group theory developed in \cite{O1} and certain results from
\cite{Osin}.

Recall that a locally compact group $G$ is said to have {\it
property (T) of Kazhdan} if the one--dimensional trivial
representation is an isolated point of the set of all irreducible
unitary representations of $G$ endowed with the Fell topology (we
refer to \cite{Kaz}, \cite{Lub} and \cite{HV} for more details).
It follows easily from the definition and Hulaniski's theorem that
every discrete finitely generated amenable group having property
(T) is finite. In contrast, we obtain the following unexpected
result in the case of weakly amenable groups.

\begin{theorem} There exists a $2$--generated infinite periodic weakly
amenable group $Q$ having property (T) of Kazhdan. In particular,
$Q$ is non--amenable.
\end{theorem}

We also consider a variant of Day's  question which goes back to
the papers \cite{Day}, \cite{Neu} and known as the so called "von
Neumann problem". Let $NF$ denote the class of all groups
containing no non--abelian free subgroups, and $AG$ denote the
class of all amenable groups. Obviously $AG\subseteq NF$ since any
non--abelian free group is non--amenable and the class $AG$ is
closed under the taking of subgroups \cite{Neu}. The question is
whether $NF=AG$.

Ol'shanskii \cite{Ols} shown that certain groups constructed by
him earlier (torsion groups with unbounded orders of elements in
which all proper subgroups are cyclic) are non--amenable and thus
the answer is negative. Further, in \cite{A} Adian proved that the
free Burnside groups $B(m,n)$ of sufficiently large odd exponent
$n$ and rank $m>1$ are non--amenable. It is a natural stronger
version of Day's question, whether the inclusion $$WA\cap
NF\subset AG$$ is true. We note that all groups constructed in
Theorem 2.1 contain non--abelian free subgroups (see Lemma 3.10
below). Furthermore, $B(m,n)\notin WA$ for any $m>1$ and any $n$
odd and large enough, as follows from the main result of
\cite{Osin2}. Thus these groups do not provide an answer. On the
other hand the negative answer is an immediate consequence of
Theorem 2.3.

\begin{corollary} There exists a finitely generated weakly amenable
non--amenable group which contains no non--abelian free
subgroups.
\end{corollary}

In conclusion we note that our construction of the group $Q$ from
Theorem 2.3 is closely related to the question whether any
finitely generated group of exponential growth is of uniform
exponential growth (see Section 4 for definitions). Originally,
this problem was formulated in \cite{GLP} and studied intensively
during the last few years (we refer to \cite{Har-book} for
survey). In \cite{Koubi}, Koubi proved that the exponential growth
rate $\omega (G)$ of every non--elementary hyperbolic group $G$
satisfies the inequality $\omega (G)>1$. On the other hand, the
following question is still open.

\begin{ques} Is the set $$\Omega _{\mathcal H}=\{ \omega (G)\; : \; G
{\rm \; is \; non-elementary\; hyperbolic }\} $$ bounded away from
the identity?
\end{ques}

In Section 4, we observe that the negative answer would imply the
existence of a finitely generated group having non--uniform
exponential growth.

\section{Non--Hopfian weakly amenable groups}

Let $F_m$ be the free group of rank $m$, $X=\{ x_1, x_2, \ldots ,
x_m\} $ a free generating set of $F_m$. We begin this section by
describing the Grigorchuk's construction of a topology on
$\mathcal G_m$, the set of all normal subgroups of $F_m$ (or,
equivalently, on the set of all group presentations with the same
generating set).

\begin{definition} The {\it Cayley graph} $\Gamma = \Gamma (G,S)$ of a
group
$G$ generated by a set $S$ is an oriented labeled 1--complex with
the vertex set $V(\Gamma )=G$ and the edge set $E(\Gamma )=G\times
S$. An edge $e=(g,s)\in E(\Gamma )$ goes from the vertex $g$ to
the vertex $gs$ and has the label $\phi (e)=s$. As usual, we
denote the origin and the terminus of the edge $e$, i.e., the
vertices $g$ and $gs$, by $\alpha (e)$ and $\omega (e)$
respectively. One can endow the group $G$ (and, therefore, the
vertex set of $\Gamma $) with a {\it length function} by assuming
$\|g\|_S$, the length of an element $g\in G$, to be equal to the
length of a shortest word in the alphabet $S\cup S^{-1}$
representing $g$.
\end{definition}

Let $N\in \mathcal G_m$. To simplify our notation we will identify
the set $X$ with the generating set of the quotient group $F_m/N$
naturally obtained from $X$. Now let $N _1, N_2$ be two normal
subgroups of $F_m$ and $G_1=F_m/N_1$, $G_2=F_m/N_2$. By $B_i(r)$,
$i=1,2$, we denote the ball of radius $r$ around the identity in
the Cayley graph $\Gamma _i=\Gamma (G_i, X)$, i.e., the oriented
labeled subgraph with the vertex set $$ V(B_i(r))=\{ g\in G_i\;
:\; \|g\|_{X_i}\le r\} $$ and the edge set $$ E(B_i(r))=\{ e\in
E(\Gamma _i)\; :\; \alpha (e)\in V(B_i(r))\; {\rm and} \; \omega
(e)\in V(B_i(r))\} .$$ One says that the groups $G_1$ and $G_2$
are {\it locally $r$--isomorphic } (being considered quotients of
$F_m$)  and writes $G_1\sim _rG_2$ if there exists a graph
isomorphism $$\iota :B_1(r)\to B_2(r)$$ that preserves labels and
orientation.

\begin{definition} For every $N\in \mathcal G_m$ and $r\in \mathbb N$,
we
consider the set $$W_r(N)=\{ L\in \mathcal G_m\; :\; F_m/N\sim _r
F_m/L\} .$$ One defines the topology on $\mathcal G_m$ by taking
the collection of the sets $W_r(N)$ as the base of neighborhoods.
\end{definition}

\begin{example} Suppose that $\{ N_i\} $ is a sequence of
normal subgroups of $F_m$ such that $N_1\ge N_2\ge \ldots $. Then
the limit of the sequence coincides with $\bigcap
\limits_{i=1}^{\infty } N_i$. Symmetrically if $N_1\le N_2\le
\ldots $, then the limit is the union $\bigcup
\limits_{i=1}^{\infty } N_i$. The proof is straightforward and is
left as an exercise to the reader.
\end{example}

We need the following result, which is proved in \cite{Osin} (up
to notation).

\begin{theorem} Suppose that $\{ N_i\} _{i \in \mathbb
N}$ is a sequence of elements of $\mathcal G_m$ which converges to
an element $N\in \mathcal G_m$. If the group $G=F_m/N$ is
amenable, then \begin{equation} \lim\limits _{i\to \infty} \alpha
(F_m/N_i, X) =0.\label{k}
\end{equation}
\end{theorem}

\begin{remark} Let $\mathcal {AG}_m$ denote the subset of all
elements $N\in \mathcal G_m$ such that the quotient group $F_m/N$
is amenable. Essentially the theorem says that the map $\alpha :
\mathcal G_m \to [0, +\infty )$ which takes each $N\in \mathcal
G_m$ to $\alpha (F_m/N, X)$ is continuous at any point $N\in
\mathcal {AG} _m$. It is not hard to see that $\alpha $ is not
continuous at arbitrary point of $\mathcal G_m$. Indeed, consider
the sequence of subgroups $N_1\ge N_2\ge \ldots $ of finite index
in $F_m$ such that
\begin{equation}
\bigcap\limits_{i=1}^\infty N_i=\{ 1\} \label{rf}
\end{equation}
(such a sequence exists since any free group is residually
finite). One can easily check that (\ref{rf}) implies
$$\lim\limits_{i\to \infty} N_i=\{1 \} $$ (see Example 3.3). Since
the group $F_m$ is non--amenable whenever $m>1$, we have $\alpha
(\{ 1\} )>0$. However, $\alpha (F_m/N_i, X)=0$ for any $i$, as the
quotient groups $F_m/N_i$ are finite (and, therefore, amenable).
\end{remark}

Now suppose that $G$ is the group defined by (\ref{GnH}). The
following four lemmas are proved under the assumptions of Theorem
2.1. Consider the homomorphism $\phi :G\to G$ induced by $\phi
(t)=t$ and $\phi (a)=\lambda (a)$ for every $a\in A$.

\begin{lemma}
The homomorphism $\phi $ is well--defined.
\end{lemma}

\begin{proof} We have to check that for any relation $R=1$ of the
group
$G$ one has $\phi (R)=1$ in $G$. There are two possibilities for
$R$.

1) First suppose that $R=1$ is a relation of the group $A$. Since
the restriction of $\phi $ to $A$ coincides with the monomorphism
$\lambda $, we have $\phi (R)=\lambda (R)=1$.

2) Assume that $R$ has the form $(\lambda (a))^t(\mu (a))^{-1}$.
Taking into account the first condition of Theorem 2.1, we obtain
$$ \phi \big( (\lambda (a))^t(\mu (a))^{-1}\big) = (\lambda \circ
\lambda (a))^t (\lambda\circ \mu (a))^{-1}= \mu \circ \lambda (a)
(\mu\circ \lambda (a))^{-1}=1.$$
\end{proof}

\begin{lemma} The map $\phi $ is surjective.
\end{lemma}

\begin{proof} Observe that $G$ is generated by $t$ and $A$. As $t\in
\phi
(G)$, it suffices to prove that $A\le \phi (G)$. Clearly we have
$\lambda (A)=\phi (A)\in \phi (G)$ and $\mu (A)=(\lambda (A))^t\in
\phi (G)$. It remains to refer to the second condition of Theorem
2.1.
\end{proof}

Let us denote by $\phi ^i$ the $i$--th power of $\phi $ and by
$N_i$ its kernel. Put $N=\bigcup\limits_{i=1}^\infty N_i$.
Obviously the group $\overline{G}=G/N$ is generated by the images
of $a$ and $t$ under the natural homomorphism $G\to \overline{G}$.
To simplify our notation we will denote these images by $a$ and
$t$ as well.

\begin{lemma} The group $\overline{G} $ is an extension of an abelian
group by a cyclic one.
\end{lemma}

\begin{proof} We denote by $B$ the kernel of the natural homomorphism
$\overline{G}\to \langle t\rangle $. Let us show that $B$ is
abelian. It is clear that $B$ is generated by the set $\{
a^{t^i}\; :\; a\in A, i\in \mathbb Z\} .$ Therefore, it is
sufficient to show that $[a^{t^i}, a^{t^j}]=1$ for any $a\in A$,
$i,j\in \mathbb Z$. Without loss of generality we can assume that
$i\ge j$. Moreover, conjugating by a suitable power of $t$, we can
assume $j=0$. In these settings, we have $$ \phi ^i([a^{t^i},a])=
[(\lambda ^i (a))^{t^i}, \lambda ^i (a)]=[\mu ^i(a), \lambda ^i
(a)]=1$$ as $A$ is abelian. Therefore, the element $[a^{t^i}, a]$
belongs to $N_i$ and thus its image in $\overline G$ is trivial.
\end{proof}

We note that in certain particular cases (including, for example,
non--Hopfian Baumslag--Solitar groups) Lemma 3.8 follows from a
result of Hirshon \cite{Hir}. As any abelian group is amenable and
the class of amenable groups is closed under group extensions,
Lemma 3.8 yields

\begin{corollary} The group $\overline{G} $ is amenable.
\end{corollary}

\begin{lemma} The group $G$ contains a non--abelian free subgroup.
\end{lemma}

\begin{proof} According to the third condition of Theorem 2.1 there
exists
an element $a\in A\setminus (\lambda (A)\cup \mu (A))$. The
elements $t$ and $a^{-1}ta$ generate the free group of rank $2$ by
the Britton lemma on HNN--extensions.
\end{proof}

\begin{proof}[Proof of Theorem 2.1.] Let us note that the sequence $\{
N_i\}$ converges to $N$. Applying Corollary 3.9 and Theorem 3.4,
we obtain $\lim\limits_{i\to\infty } \alpha (F/N_i, X)=0$. On the
other hand, $F/N_i\cong G$, this means that $\alpha (G)=0$, i.e.,
$G$ is weakly amenable. Finally, $G$ is non--amenable according to
Lemma 3.10.
\end{proof}

\section{Common quotient groups of all non--elementary hyperbolic
groups.}

Let us recall just one of a number of equivalent definitions of
hyperbolicity. A group $G$ with a finite generating set $X$ is
{\it hyperbolic} (in the sense of Gromov) if its Cayley graph
$\Gamma =\Gamma (G,X)$ is a hyperbolic space with respect to the
natural metric. This means that any geodesic triangle in $\Gamma $
is $\delta $--thin for a fixed constant $\delta $, i.e., each of
its sides belongs to the closed $\delta $--neighborhood of the
union of other two sides.

It has been mentioned by Gromov \cite{MG} (see also \cite{HV}),
that an element $g$ of infinite order in a hyperbolic group $G$ is
contained in a unique maximal elementary subgroup $E_G(g)$ ({\it
elementary closure of $g$}). For a subgroup $H$ of a hyperbolic
group $G$, its elementarizer $E_G(H)$ is defined as $\cap E_G(h)$,
where $h$ ranges over all elements of infinite order in $H$. If
the subgroup $H$ is non--elementary, $E_G(H)$ is the unique
maximal finite subgroup of $G$ normalized by $H$ \cite[Proposition
1]{O1}; notice also that $E_G(G)$ is the kernel of the action of
$G$ on the hyperbolic boundary $\partial G$ induced by left
multiplication on $G$.

The following is the simplification of Theorem 2 from \cite{O1}
(see also \cite[Lemma 5.1]{OJA}).

\begin{lemma} Let $H_1, \ldots , H_k$ be non--elementary subgroups of
a
hyperbolic group $G$ such that $E_G(H_1)=\ldots =E_G(H_k)=1$. Then
there exists a non--elementary hyperbolic quotient $K$ of $G$ such
that the image of each subgroup $H_1, \ldots , H_k$ under the
natural epimorphism $G\to K$ coincides with $K$.
\end{lemma}

\begin{corollary} Let $P_1, \ldots , P_k$ be non--elementary
hyperbolic
groups. Then there exists a non--elementary hyperbolic group $Q$
that is a homomorphic image of $P_i$ for every $i=1, \ldots , k$.
\end{corollary}

\begin{proof} The proof can be extracted from the one of Theorem 2 in
\cite{OJA}. Here we provide it for convenience of the reader. Let
us set $H_i=P_i/E_{P_i}(P_i)$. Clearly $E_{H_i}(H_i)=1$, as
$E_{P_i}(P_i)$ is the maximal normal finite subgroup of $P_i$.
Moreover, since any quotient of a non--elementary hyperbolic group
modulo a finite normal subgroup is also a non--elementary
hyperbolic group \cite[Corollary 23(ii)]{GH}, it follows that
$H_i$ is non--elementary hyperbolic. Now we take the free product
$$G=H_1\ast \ldots \ast H_k.$$ It is easy to check that
$E_G(H_i)=1$ for every $i$ as there are no finite subgroups of $G$
normalized by $H_i$. It remains to apply Lemma 4.1.
\end{proof}

We need one more lemma (the proof can be found in \cite{O1}).

\begin{lemma} Let $G$ be a non--elementary hyperbolic group, $g$ an
element of $G$. Then there exists $N\in \mathbb N$ such that the
quotient group of $G$ modulo the normal closure of $g^N$ is
non--elementary and hyperbolic.
\end{lemma}

Now we are going to describe the main construction of the present
section.

\begin{theorem} There exists a $2$--generated infinite periodic group
$Q$
such that for every non--elementary hyperbolic group $H$, there is
an epimorphism $\rho : H\to Q$.
\end{theorem}

\begin{proof}  Since any hyperbolic group is finitely presented, the
set
of all non--elementary hyperbolic groups is countable. Let us
enumerate this set $G_1, G_2, \ldots $ and elements of the first
group $G_1=\{ g_1, g_2, \ldots \}$. Consider the following
diagram, which is constructed by induction. $$
\begin{array}{cccccccccccccc}
G_1 &&&& G_2 &&\ldots &&& G_{k} &&&& \ldots \\
\Big\downarrow\vcenter{\rlap{$\scriptstyle{\pi _1}$}} && &&
\Big\downarrow\vcenter{\rlap{$\scriptstyle{\pi _2}$}} && && &
\Big\downarrow\vcenter{\rlap{$\scriptstyle{\pi _{k}}$}} &&&& \\
Q_1 & \stackrel{\psi _1}{\longrightarrow} & R_1 & \stackrel{\phi
_2}{\longrightarrow} & Q_2 & \stackrel{\psi _2}{\longrightarrow} &
\ldots &  R_{k-1} & \stackrel{\phi _k}{\longrightarrow} & Q_{k} &
\stackrel{\psi _{k}}{\longrightarrow} & R_k & \stackrel{\phi
_{k+1}}{\longrightarrow} & \ldots
\end{array}
$$ Suppose $G_1=Q_1$ and let $\pi _1 $ denote the corresponding
natural isomorphism. Assume that we have already defined the
groups $Q_{i}$, $R_{i-1}$ and homomorphisms $\phi _i: R_{i-1}\to
Q_i$, $\psi _{i-1}: Q_{i-1}\to R_{i-1}$ for all $i\le k$. Denote
by $\tau _{k}: G_1\to Q_{k}$ the composition $ \phi _k\psi _{k-1}
\ldots \phi _2\psi _1 \pi _1$ and by $\bar g_i$ the image of $g_i$
in $Q_{k}$ under $\tau _k$. According to Lemma 4.3, there exists
$N_i\in \mathbb N$ such that the quotient $Q_{k}/\langle\bar
g_i^{N_i}\rangle ^{Q_{k}}$ is a non--elementary hyperbolic group.
We set $$R_{k}=Q_{k}/\langle\bar g_i^{N_i}\rangle ^{Q_{k}}$$ and
denote by $\psi _{k}$ the natural homomorphism from $Q_{k} $ to
$R_{k}$. Further, by Corollary 4.2, there is a non--elementary
hyperbolic group $Q_{k+1}$ such that there exist epimorphisms
$$\phi _{k+1}: R_{k}\to Q_{k+1}\;\;\; {\rm and}\;\;\; \pi _{k+1} :
G_{k+1}\to Q_{k+1}. $$ The inductive step is completed.

Let us denote by $U_k$ the kernel of $\tau _k$. Evidently we have
$\{ 1\} = U_1 \le U_2\le \ldots $. Set
$U=\bigcup\limits_{i=1}^\infty U_i$ and consider the quotient
group $Q=G_1/U$. Note that one can assume $G_1$ to be 2--generated
without loss of generality. Further, $Q$ is a quotient group of
$Q_i$ for all $i$, hence $Q$ is a quotient of $G_i$ for all $i$.
The periodicity of $Q$ follows directly from our construction. It
remains to show that $Q$ is infinite. To do this, let us suppose
that $Q$ is finite. Then $Q$ is finitely presented. Therefore,
$Q_i$ is a quotient group of $Q$ for all $i$ big enough. In
particular, $Q_i$ is elementary whenever $i$ is sufficiently big
and we get a contradiction. The theorem is proved.
\end{proof}

Let us denote by $\mathcal {H}_m$ the subset of all $N\in \mathcal
G_m$ such that $F_m/N$ is non-elementary and hyperbolic. Recall
also that $\mathcal {AG}_m$ denotes the subset of all $N\in
\mathcal G_m$ such that $F_m/N$ is amenable. The following two
observations from \cite{Osin} plays the crucial role in the
studying of the group $Q$.

\begin{theorem} For every $m\ge 2$, the intersection of the closure of
$\mathcal {H}_m $ (with respect to the Cayley topology on
$\mathcal G_m$) and $\mathcal {AG}_m$ is non--empty.
\end{theorem}

\begin{lemma} Suppose that $G$ is a finitely generated group and $\phi
:
G\to P$ is a surjective homomorphism onto a group $P$. Then
$\alpha (G) \ge \alpha (P)$.
\end{lemma}

Now we want to show that the group $Q$ from Theorem 4.4 has all
properties listed at Theorem 2.3.

\begin{proof}[Proof of Theorem 2.3]. By Theorem 3.4 and Theorem 4.5,
there
is a sequence of elements $N_i\in \mathcal H_2$, $i\in \mathbb N$,
such that
\begin{equation}
\lim\limits_{i\to\infty }\alpha (F_2/N_i)=0. \label{GiXi}
\end{equation}
Let us denote by $G_i$ the quotient group $F_2/N_i$. According to
Theorem 4.4, there exists an epimorphism $\rho _i: G_i\to Q$ for
every $G_i$. Combining Lemma 4.6 and (\ref{GiXi}), we obtain
$\alpha(Q)=0.$ As is well known, there are non--elementary
hyperbolic groups having property $(T)$ of Kazhdan (for instance,
uniform lattices in $Sp(n,1)$). Since the class of Kazhdan groups
is closed under the taking of quotients, the group $Q$ has the
property $T$. Recall that any discrete amenable Kazhdan group is
finite; taking into account the infiniteness of $Q$, we conclude
that $Q$ is non--amenable.
\end{proof}

In conclusion we discuss certain relations with growth functions
of hyperbolic groups. The {\it growth function} $\gamma_G^X :
\mathbb N \longrightarrow \mathbb N$ of a group $G$ generated by a
finite set $X$ is defined by the formula $$\gamma _G^X(n)=card\;
\{ g\in G\; :\; ||g||_X\le n\} ,$$ where $||g||_X$ denotes the
word length of $g$ relative to $X$. The {\it exponential growth
rate} of $G$ with respect to $X$ is the number $$\omega (G,X) =
\lim_{n \to \infty} \sqrt[n]{\gamma _G^X(n)} .$$ The above limit
exists by submultiplicativity of $\gamma_G^X $. The group $G$ is
said to be of {\it exponential growth } (respectively of {\it
subexponential growth }) if $\omega (G,X)>1$ (respectively $\omega
(G,X)=1$) for some generating set $X$.

It is easy to see that above definitions are independent of the
choice of a generating set in $G$. Let us consider the quantity $$
\omega (G) = \inf\limits_X \omega (G,X),$$ where the infimum is
taken over all finite generating sets of $G$. One says that $G$
has uniform exponential growth if
\begin{equation}
\omega (G)>1. \label{w}
\end{equation}
It is an open question whether any group of exponential growth
satisfies (\ref{w}). We observe that Theorem 4.4 provides an
approach to the solution of this problem.

\begin{lemma} Let $G$ be a group generated by a finite set $X$ and
$\phi :
G\to P$ be an epimorphism. Then $\omega (G,X)\ge \omega (P, \phi
(X))$.
\end{lemma}

\begin{proof} This observation is well known and quite trivial. The
proof
follows easily from the inequality $\| g\| _X\ge \| \phi (g) \|
_{\phi (X)}$. We leave details to the reader.
\end{proof}

Obviously Lemma 4.7 and Theorem 4.2 yield the following.

\begin{corollary} Suppose that for every $\varepsilon >0$, there exists
a
non--elementary hyperbolic group $H$ such that $\omega (H)<
1+\varepsilon $. Then the group $Q$ from Theorem 4.1 has
non--uniform exponential growth, i.e., $\omega (Q,X)>1$ for any
finite generating set $X$ of $Q$ but $\omega (Q)=1$.
\end{corollary}

\bibliographystyle{amsalpha}

\end{document}